\documentclass{article}
\usepackage{microtype}
\usepackage[style=numeric-comp, bibencoding=utf8]{biblatex}
\renewbibmacro{in:}{}
\usepackage{amsmath,amssymb,amsthm}
\usepackage{xspace}
\usepackage[breaklinks]{hyperref}
\usepackage[anythingbreaks]{breakurl}
\usepackage[]{todonotes}		
\usepackage{mathtools}
\mathtoolsset{showonlyrefs}   
 \usepackage{colortbl} 
 \usepackage{enumerate}
\newcommand{\dd}{\mathop{}\!\mathrm{d}}
\newcommand{\dv}[2]{\frac{\dd #1}{\dd #2}}

\let\epsilon\varepsilon

\newtheorem{thm}{Theorem}[section]

\newtheorem{lem}[thm]{Lemma}
\newtheorem{prop}[thm]{Proposition}
\theoremstyle{definition}
\newtheorem{defn}[thm]{Definition}  

\theoremstyle{remark}
\newtheorem{rem}[thm]{Remark}
\numberwithin{equation}{section}
\newcommand{\Holder}{H\"older\xspace}
\newcommand{\HKR}{Hmidi--Keraani--Rousset}
\newcommand{\WX}{Wu--Xue}

\newcommand{\myR}{\mathcal{R}_{2\alpha}}

\newcommand{\functioninput}{\bullet}
\newcommand{\coloneq}{\mathrel{\mathop:}=}

\begin{document}

\title{Temperature patches for the subcritical Boussinesq--Navier--Stokes System with no diffusion}
\author{Calvin Khor\footnote{Corresponding author.}
  \     and  
  Xiaojing Xu}
\date{}
\maketitle
\abstract{In this paper, we prove that temperature patch solutions to the subcritical Boussinesq--Navier--Stokes System with no diffusion preserve the \Holder regularity of their boundary for all time, which generalises the previously known result by F. Gancedo and E. Garc\'ia-Ju\'arez [\emph{Annals of PDE}, 3(2):14, 2017] to the full range of subcritical viscosity. }\\[1em]

\section{Introduction}This paper studies temperature patch solutions of the following initial-value problem for the subcritical Boussinesq--Navier--Stokes equations, $\alpha \in (\frac12,1)$:
\begin{align}
\left\{\begin{array}{rcll}
    u_t + u\cdot \nabla u + \Lambda^{2\alpha} u &=& \nabla p + \theta e_2 , & (t,x)\in\mathbb R_+ \times \mathbb R^2,\\
    \nabla\cdot u &=& 0, \\
    \theta_t + u\cdot \nabla \theta &=& 0, \\
    u|_{t=0} &=& u_0,\\
    \theta|_{t=0} &=& \theta_0.
\end{array}\right.\label{the-eqn}
\end{align}
Here, $\Lambda^{2\alpha}=(-\Delta)^\alpha$ is a fractional Laplacian  on $\mathbb R^2$ defined initially for Schwartz functions as a Fourier multiplier $\widehat{\Lambda^{2\alpha}f}(\xi) = |2\pi \xi|^{2\alpha}\hat f(\xi)$. Also,  $e_2 := (0,1)^T\in\mathbb R^2$,  $u=u(t,x)=(u_1(t,x),u_2(t,x))^T \in \mathbb R^2$ is the velocity of the fluid, $p=p(t,x)\in\mathbb R$ is the pressure of the fluid, $\theta=\theta(t,x)\in\mathbb R$ is the temperature of the fluid, and $u_0,\theta_0$ are the initial data of the system. 

 This active scalar transport system for $\theta$ arises as a natural generalisation of the Boussinesq--Navier--Stokes system, where the dissipation is the full Laplacian (corresponding to $\alpha = 1$). Details of the physics described by the Boussinesq--Navier--Stokes system can be found in \cite{pedlosky1971geophysical}. Actually, \eqref{the-eqn} is a special case of a two parameter family of equations where there is also a diffusion term $\Lambda^{2\beta}\theta$ in the temperature equation; see for instance the papers \cite{xu2010global}, \cite{miao2011global}, \cite{stefanov2019global}, \cite{wu2014wellposedness}, 
  and citations within. Only local results for the zero viscosity and zero diffusion case are known, such as those in \cite{chae_nam_1997}. 
 
 In addition, some authors consider the opposite situation, where $\alpha = 0$ and $\beta>0$ (called the Euler--Boussinesq system) such as \cite{hmidi2010EBfracdissipation}, \cite{hmidi2010globalEBS}, and \cite{wuxue2012global}, but to the best of the authors' knowledge, there are no other papers currently available studying temperature patches for \eqref{the-eqn} with $\alpha\in(\frac12,1)$. 
 
 We consider solutions in the following sense:
 \begin{defn}\label{defn-solution}
We say that $(u,\theta)\in L^2_{\text{loc}}([0,\infty]\times \mathbb R^2;\mathbb R^2)\times L^2_{\text{loc}}([0,\infty]\times \mathbb R^2)$ with $\nabla\cdot u = 0$ is a solution to \eqref{the-eqn} if for every divergence-free $\phi\in C^\infty_c((0,T)\times \mathbb R^2;\mathbb R^2)$, and for every $\psi\in C^\infty_c((0,T)\times \mathbb R^2),$
\begin{align}
    \int_0^\infty \int_{\mathbb R^2} u\cdot  \phi_t + u\otimes u : \nabla \phi - u \cdot \Lambda^{2\alpha} \phi - \theta \phi\cdot e_2  \dd x \dd t = 0,\\
    \int_0^\infty \int_{\mathbb R^2} \theta \psi_t + \theta u \cdot \nabla \psi \dd x \dd t = 0.
\end{align}
\end{defn}

 By a `temperature patch solution', we mean a solution of \eqref{the-eqn} where $\theta$ is an indicator function for each time $t\ge 0$. These are similar to the sharp front solutions of the SQG equation and their variants which have been and are extensively studied 
 \cite{rodrigo2004vortex},  \cite{rodrigo2005evolution}, \cite{fefferman2011analytic}, \cite{fefferman2012almost}, \cite{fefferman2012spine}, \cite{gancedo2008existence}, \cite{gancedo2016arXiv160506663C},  \cite{chae2011inviscid}, \cite{chae2012generalized}, \cite{khor2020local}, \cite{khor2020sharp}, \cite{hunter2018global}, \cite{hunter2018local}, \cite{hunter2018regularized}, \cite{hunter2019contour}, 
 and the more classical theory of vortex patches for the 2D Euler equation as laid out in \cite{majda2002}, and also \cite{constantinwu1995inviscid}, \cite{constantinwu1996inviscid}, and \cite{JIU20121367}. The Boussinesq system also supports initial data of `Yudovich' type. We mention a recent paper  \cite{melkemi2020local} working on local-in-time Yudovich solutions to the full inviscid equation. The case of critical diffusion was studied in \cite{zerguine2015regular}. Similar problems can be and have been studied for other systems with a transport equation with no diffusion \cite{cordobaMR2753607}, \cite{gancedo2019global} \cite{gancedo2018globalInhomNS}, and multiple patch solutions for SQG have also been studied \cite{hunter2019twofront}, \cite{gancedo2018arXiv181100530G}. Finally we mention some recent work \cite{gancedo2020regularity} on piecewise H\"older solutions to the 3D analogue of the system \eqref{the-eqn} with $\alpha=1$.
 
Global unique classical solutions to \eqref{the-eqn} with $\alpha=1$  were shown to exist in $H^k$ spaces in \cite{chae2006global} and \cite{hou05globalwell-posedness}; global (weak) solutions in $(L^2\cap B^{-1}_{\infty,1})\times  B^0_{2,1}$ were studied in \cite{abidi2017global}, and global (weak) solutions in $L^2\times L^2$ were proven to exist in \cite{hmidi2007zerodiff} with uniqueness proven in \cite{danchin2008theoremesdeLeray}. In the case of temperature patches, Abidi and Hmidi \cite{abidi2007JDE} proved the persistence of $C^1$ regularity of the boundary, and Danchin and Zhang \cite{danchin2017geometrical} improved this result to $C^{1+\epsilon}$ regularity. In Gancedo and Garc\'ia--Ju\'arez \cite{gancedo2017global}, a second proof of Danchin and Zhang's result was given, and they improved the result to $C^{2+\epsilon}$ persistence, in particular implying that the curvature of the patch remains bounded.

In \cite{hmidi2009global}, the critical equation $\alpha=1/2$ has been studied in a low regularity scenario. However, the well-posedness result there requires initial data $\theta_0\in B^0_{\infty,1}\hookrightarrow C^0_b$ which falls short of allowing $L^\infty$ data, of which temperature patches are a special case. Our work shows that this seems to be a special feature of the critical $\alpha=1/2$ case, as $L^\infty$ data is allowed for all $\alpha\in(1/2,1].$ This is a curious property, as the the transport evolution of $\theta$ gives that weak solutions cannot increase $L^p$ norms of $\theta$. At the same time, there was some foreshadowing, as the method of Gancedo and Garc\'ia-Ju\'arez which relies on the explicit formula for the heat kernel can only be replicated for the critical case $\alpha=1/2$, and $e^{-t\Lambda}\nabla f $ has roughly the same regularity as $f$ (as opposed to $e^{t\Delta}\nabla f$ being better behaved than $f$: see \cite{gancedo2017global} for details.)

Our main technical result is the following existence and uniqueness result:
\begin{thm}\label{main}
    Suppose that $\alpha \in (\frac12, 1)$,  $u_0 \in H^1\cap \dot{W}^{1,p}$ for some $p\in(2,\infty)$, $\nabla\cdot u_0=0$, and $\theta_0\in L^1\cap L^\infty$. Then there is a unique global solution $(u,\theta)$ in the sense of Definition \ref{defn-solution} to \eqref{the-eqn} such that for each $\rho\in[1, \min(2,p/2))$ and for each $\alpha'\in[0 , \alpha)$,
    \begin{align}
        u\in L^\infty_t W^{1,p} \cap L^{\rho}_t C^{2\alpha'},\ \text{ and } \ 
        \theta \in L^\infty_t L^1\cap L^\infty.
    \end{align}
\end{thm}
Using Theorem \ref{main}, for any $\alpha'<\alpha$, we can control the $C^{2\alpha'}$ boundary regularity of temperature patch solutions to \eqref{the-eqn}:
\begin{thm}\label{cor-Boundary-reg-intro-ver}
Suppose that $\alpha\in(\frac12 , 1)$, $u_0\in H^1\cap \dot{W}^{1,p}$ for some $p\in(2,\infty]$, $\nabla\cdot u_0=0$, and $\theta_0 = \mathbf 1_{D_0}$ be an indicator function of a simply connected set $D_0$ with $D_0\in C^{2\alpha'}$ for some $\alpha'<\alpha$. Then the global solution  $\theta$ to \eqref{the-eqn} remains an indicator, $\theta(t)= \mathbf 1_{D(t)}$ where $\partial D(t) \in L^\infty_t C^{2\alpha'}$ for all $t\ge 0$.
\end{thm}

The main method of proof is the use of the special structure of the equation by the study of $\Gamma := \omega - \myR \theta$, where $\myR$ is a smoothing operator. It turns out that it is easier to study this combination of terms than the vorticity $\omega=\nabla^\perp \cdot u$ by itself. This is the method used in \cite{hmidi2009global}, but since $\alpha>1/2$ the proof is more streamlined.

Theorem \ref{main} raises the following interesting questions: (a) Is it possible to control the curvature for $\alpha\in(1/2,1)$, as it is possible in $\alpha=1$? (b) can the critical equation $\alpha=1$ support unique temperature patch solutions, and what regularity of their boundary is preserved?

The remainder of the paper is organised as follows. In Section \ref{sectionPrelim}, we list the notation that we use for function spaces and inequalities. In Section \ref{sectionMyR}, we explain how introducing the term $\Gamma$ leads to better estimates. In Section \ref{sectionBasic}, we give some easy a priori estimates from the equation obtained by classical means. In Section \ref{sectionAPrioriVort}, we use the equation for $\Gamma$ to derive better a priori estimates for the vorticity (and hence the velocity). In Section \ref{sectionUniqueness}, we use the Osgood Lemma to prove uniqueness of solutions. In Section \ref{sectionExistence}, we use the quantitative bound from the Osgood Lemma to show existence of solutions, and also present the proof for conservation of H\"older regularity of temperature patch boundaries.
\section{Preliminaries}
\label{sectionPrelim}
We follow the notation of \cite{hmidi2009global} as follows.
\begin{itemize}
    \item  We write $\phi_k(t)$ to denote any function of the form
\[ \phi_k(t) = C_0 \underbrace{\exp\exp\dots\exp}_{k\text{ times}}(C_1 t).\]
\item We write $A\lesssim B$ to mean that $|A|\le CB$ for some constant $C$ that does not depend on $A,B,$ or any other variable under consideration. We write $A\lesssim_{\phi_1,\dots,\phi_N} B$ to emphasise that the implicit constant $C$ depends on the $N$ quantities $\phi_1,\dots,\phi_N$.
\item When $(X,\|\functioninput\|_X)$ is a Banach space, we will write for $p\in[1,\infty]$
$$ L^p_t X$$
to denote the Bochner space $L^p(0,t; X)$; a function $f=f(t)$ with values in $X$ is in $L^p_t X$ if
\begin{align}
\|f\|_{L^p_t X} := \left \| \|f(t)\|_X \right\|_{L^p_t} < \infty.    
\end{align}
In a norm, we also abusively adopt the notation that a subscripted variable $q$,  like $\ell^r_q$ below, means that the norm is to be taken with respect to $q$. This should not cause confusion with the above convention for time integrals.
See for instance \cite{evans1998partial} or \cite{rudin2006functional} for details on Bochner spaces.
\end{itemize}
\subsection{Littlewood--Paley Decomposition and Function spaces} We fix our notation for Littlewood--Paley decompositions here; see 
\cite{bahouri2011fourier}, 
\cite{chemin1998perfect}, 
\cite{peetre1976new}, 
\cite{tao2006nonlinear}, 
or 
\cite{sawano2018theory} for more details about Littlewood--Paley, Besov spaces, and Tilde spaces.

There exist two radial non-negative functions $\chi \in \mathcal{D}\left(\mathbb{R}^{d}\right)$ and $\varphi \in \mathcal{D}(\mathbb{R}^{d} \backslash\{0\})$ such that
\begin{enumerate}[(i)]
    \item $\chi(\xi)+\sum_{q \geqslant 0} \varphi\left(2^{-q} \xi\right)=1 ; \forall q \geqslant 1, \operatorname{supp} \chi \cap \operatorname{supp} \varphi\left(2^{-q}\right)=\emptyset$,
    \item  $\operatorname{supp} \varphi\left(2^{-j} \functioninput \right) \cap \operatorname{supp} \varphi\left(2^{-k} \functioninput\right)=\emptyset,$ if $|j-k| \geqslant 2$.
\end{enumerate} 
For every $v \in \mathcal{S}^{\prime}\left(\mathbb{R}^{d}\right)$ we define Fourier multipliers $\Delta_q, S_q$ by
\[
\Delta_{-1} v=\chi(\mathrm{D}) v; \ \forall q \in \mathbb{N}, \ \Delta_{q} v=\varphi\left(2^{-q} \mathrm{D}\right) v; \ \text { and } \ S_{q}v=\!\!   \sum_{-1 \le  p \le q-1}\! \Delta_{p}v.
\]

\begin{lem}[Bernstein Inequalities] There exists $a$ constant $C$ such that for $q, k \in \mathbb{N}, 1 \leqslant a \leqslant b$ and for $f \in L^{a}(\mathbb{R}^{d})$
\begin{align}
    \sup _{|\alpha|=k}\left\|\partial^{\alpha} S_{q} f\right\|_{L^{b}} &\leqslant C^{k} 2^{q\left(k+d\left(\frac{1}{a}-\frac{1}{b}\right)\right)}\left\|S_{q} f\right\|_{L^{a}}, \text{ and}\\
C^{-k} 2^{q k}\left\|\Delta_{q} f\right\|_{L^{a}} &\leq \sup _{|\alpha|=k}\left\|\partial^{\alpha} \Delta_{q} f\right\|_{L^{a}} \le  C^{k} 2^{q k}\left\|\Delta_{q} f\right\|_{L^{a}}.
\end{align}
\end{lem}

\subsubsection{Besov and Tilde Spaces}
The Besov space $B^s_{p,q}(\mathbb R^2)$ for $p,q\in[1,\infty],\ s\in\mathbb R$ is the space of distributions $u$ such that  
\[
\|u\|_{B_{p, r}^{s}}:=\left\|2^{q s}\left\|\Delta_{q} u\right\|_{L^{p}}\right\| _{\ell^r_q}< \infty.
\]
In a norm, a subscripted variable $q$  like $\ell^r_q$ above means that the norm is to be taken with respect to $q$. Since we will only consider functions on $\mathbb R^2$, we will only write $B^s_{p,q}$ (and similarly $L_t^p B^s_{p,q}$ instead of $L_t^p B^s_{p,q}(\mathbb R^2)$). Also write $H^s := B^s_{2,2}$. The Besov spaces trivially satisfy
\begin{align}
    s_1 \ge s_2 &\implies B^{s_1}_{p,r} \hookrightarrow B^{s_2}_{p,\tilde r} ,\quad \forall p,r,\tilde r \in [1,\infty], 
    \\
    r_1 \le r_2 & \implies B^s_{p,r_1} \hookrightarrow B^s_{p,r_2} , \quad \forall s\in\mathbb R,\, p \in [1,\infty],
\end{align}
and there is also the analogue of Sobolev embedding in 2D: 
\begin{align}
    B^{s_0+s}_{p,r} \hookrightarrow B^{s_0}_{p*,r}, \quad \frac1{p_*} = \frac1p - \frac s2, \quad \forall s_0\in\mathbb R,p,r\in[1,\infty].
\end{align}
In addition, we have the embeddings for any $p\in[1,\infty]$,
$$B^0_{p,1} \hookrightarrow L^p \hookrightarrow B^0_{p,\infty}, \quad k\in\mathbb N_0\implies B^k_{p,1} \hookrightarrow W^{k,p} \hookrightarrow B^k_{p,\infty} .$$ 
We say that a function $u=u(x,t)$ belongs to the `Tilde space' $\widetilde L^p_t B^s_{p,q}$ if
\[ \|u\|_{\widetilde L^p_t B^s_{p,r}} := \left\|2^{sq} \left\| \left\| \Delta_q u(x,t) \right\|_{L^p_x}\right\|_{L^p_t} \right\|_{\ell^r_q}. \]
In comparison with
\[ \|u\|_{L^p_t B^s_{p,r}} =  \left\|  \left\| u \right\|_{(B^s_{p,r})_x}  \right\|_{L^p_t} = \left\| \left\| 2^{sq}\left\| \Delta_q u(x,t) \right\|_{L^p_x}\right\|_{\ell^r_q}  \right\|_{L^p_t},\]
From the generalised Minkowski inequality $\big \|\|f\|_{L^p_x}\big \|_{L^q_y} \le \big \|\|f\|_{L^q_y}\big \|_{L^p_x}$ if $q\ge p$ and embeddings between Besov spaces, we have the following relations:
\begin{align}
    L_{T}^{\rho} B_{p, r}^{s} \hookrightarrow \tilde{L}_{T}^{\rho} B_{p, r}^{s} \hookrightarrow L_{T}^{\rho} B_{p, r}^{s-\varepsilon},\ & \text { if } r \ge\rho, 
\\
 L_{T}^{\rho} B_{p, r}^{s+\varepsilon} \hookrightarrow \tilde{L}_{T}^{\rho} B_{p, r}^{s} \hookrightarrow L_{T}^{\rho} B_{p, r}^{s},\ & \text { if } \rho \ge r,
\end{align} 
where $\varepsilon>0$ is arbitrarily small. In particular $\tilde{L}_{T}^{r} B_{p, r}^{s} = L_{T}^{r} B_{p, r}^{s}$.
\section{Study of $\myR$ and some commutators}
\label{sectionMyR}
\subsection{Introduction of the $\Gamma = \omega - \myR \theta$ term}
We use the technique of \cite{hmidi2009global}, \cite{wuxue2012global}, \cite{miao2011global}, and other papers of introducing an auxillary quantity $\Gamma$ determined from the equation, that has better regularity properties than the original functions under consideration. The $\omega:=\nabla^\perp\cdot u$ equation is obtained by applying the $\nabla^\perp\cdot := \binom{-\partial_2}{\partial_1}\cdot $ operator to the first equation of \eqref{the-eqn}. By writing $\partial_1 \theta = \Lambda^{2\alpha}\Theta$, i.e. $\Theta = \myR \theta := \Lambda^{1-2\alpha}\mathcal R_1 \theta$, the $\omega$ equation takes the form
\[ (\partial_t + u\cdot\nabla + \Lambda^{2\alpha} ) \omega - \Lambda^{2\alpha} \Theta = 0 .\]
By adding $-(\partial_t + u\cdot \nabla)\Theta$ to both sides of the equation, we obtain an equation for $\Gamma = \omega - \Theta $ which has a commutator:
\begin{align}
    (\partial_t + u\cdot\nabla + \Lambda^{2\alpha} )\Gamma = [\myR,u\cdot\nabla ]\theta   .
\end{align}
This structure will be key in deriving the a priori estimates for $u$.
An important lemma for the study of the critical dissipation case is Lemma 3.3 in \HKR's paper \cite{hmidi2009global}. The generalisation of Lemma 3.3 (ii) is the following result of \WX{} \cite{wuxue2012global}:
\begin{prop}[Proposition 4.2 of \cite{wuxue2012global}]\label{wx-prop42}
     Let $\beta \in [1,2),(p, r) \in[2, \infty] \times[1, \infty), u $ be a smooth divergence-free vector field of $\mathbb{R}^{n}$ $(n \ge 2) $ with vorticity $ \omega $ and $ \theta $ be a smooth scalar function. Then we have that for every $ s \in( \beta-2, \beta),$
$$
\left\|\left[\mathcal{R}_{\beta}, u \cdot \nabla\right] \theta\right\|_{B_{p, r}^{s}} \lesssim_{s, \beta}\|\nabla u\|_{L^p}\left(\|\theta\|_{B_{\infty, r}^{s+1-\beta}}+\|\theta\|_{L^{2}}\right)
.   $$
Besides, if $p=\infty,$ we also have
$$
\left\|\left[\mathcal{R}_{\beta}, u \cdot \nabla\right] \theta\right\|_{B_{\infty,r}^{s}} \lesssim_{s, \beta}\left(\|\omega\|_{L^{\infty}}+\|u\|_{L^{2}}\right)\|\theta\|_{B_{\infty, r}^{s+1-\beta / 2}}+\|u\|_{L^{2}}\|\theta\|_{L^{2}}.
$$

\end{prop}

 We will similarly generalise Lemma 3.3 (i) of \cite{hmidi2009global} by using Bony's decomposition, as follows.
\begin{thm}\label{thm-gen-HKRlem3.3}
    For any  $s\in ( 0 ,2\alpha)$ and any smooth  $v,\theta$ with $\nabla\cdot v = 0$, 
    \[ \|[\myR,v]\theta\|_{H^s} \lesssim_s \|\nabla v\|_{L^2} \|\theta\|_{B^{s-2\alpha}_{\infty,2}} + \|v\|_{L^2} \|\theta\|_{L^2}.\]
\end{thm}
We prove this result using the following lemma and the equivalent of \HKR{}'s Proposition 3.1 (\WX's Proposition 4.1). This lemma is also quoted in \cite{hmidi2010globalEBS}. We give its short proof here, which   is a variant of Lemma 2.97 of \cite{bahouri2011fourier}.
\begin{lem}[Lemma 3.2 of \cite{hmidi2009global}]\label{lem-hkr3-2}
    Let $p\in[1,\infty]$, and let $f,g$, and $h$ be three functions such that $\nabla f\in L^p$, $g\in L^\infty$, and $xh \in L^1$. Then
    \begin{align}
        \| h * (fg) - f(h*g)\|_{L^p} \le \|xh(x)\|_{L^1_x   } \|\nabla f\|_{L^p} \|g\|_{L^\infty}.
    \end{align}
\end{lem}
\begin{proof}
    By a direct computation and the fundamental theorem of calculus, setting $F:= h * (fg) - f(h*g)$,
    \begin{align}
        F(x)&= \int_{\mathbb R^n} h(x-y) \big(f(y) - f(x)\big) g(y) \dd y 
        \\ &= \int_0^1\int_{\mathbb R^n} h(x-y) \nabla f((1-\tau)x + \tau y) \cdot (y-x) g(y) \dd y\dd \tau
        \\ &= -\int_0^1 \int_{\mathbb R^n} h(z) \nabla f(x - \tau z) \cdot z g(x-z)  \dd z \dd \tau.
    \end{align}
Therefore, \Holder's inequality and translation invariance of the Lebesgue measure gives
\begin{align}
    \|F\|_{L^p} &\le \int_0^1 \int_{\mathbb R^n} |zh(z)| \| \nabla f(x - \tau z) g(x - z)\|_{L^p_x} \dd z\dd \tau
    \\ &\le \int_{\mathbb R^n} |zh(z)|  \int_0^1 \|\nabla f(x-\tau z)\|_{L^q_x} \dd\tau  \|g(x-z)\|_{L^r_x} \dd z
    \\ &\le \|zh(z)\|_{L^1_z} \|\nabla f\|_{L^q} \|g\|_{L^r},
\end{align} for any $\frac1q+\frac1r = \frac1p$. In particular, setting $q=p,\ r=\infty$, we obtain the claimed result. \end{proof}

\begin{prop}[Proposition 4.1 of \cite{wuxue2012global}]\label{prop-wx-4-1}
    Let $j\in\mathbb N$, $\alpha \in (\frac12,1)$, and $\myR = \Lambda^{1-2\alpha}\mathcal R$. Then:
\begin{enumerate}
    \item For every $p\in(1,\infty)$ and  $q>p$ defined by $\frac1q = \frac1p - \frac{2\alpha-1}n $, $\myR$ is a bounded map $L^p\to L^q$.
    \item Let $\chi\in\mathcal D(\mathbb R^n)$. Then for each $(p,s)\in [1,\infty]\times (2\alpha - 1,\infty)$ and $f\in L^p(\mathbb R^n)$, \begin{align}
        \| |\nabla |^s \chi(2^{-j} |\nabla | ) \myR f\|_{L^p} \le 2^{j(s+1-2\alpha)} \|f\|_{L^p} .
    \end{align}
    Moreover, $|\nabla |^s \chi(2^{-j} |\nabla | )$ is a convolution operator with kernel $K$ satisfying 
    \[ |K(x)| \lesssim \frac1{(1+|x|)^{n+s+1-\beta}}.\]
    \item Let $\mathcal O$ be an annulus centered at the origin. Then there exists $\phi\in\mathcal S(\mathbb R^n)$ with spectrum supported away from $0$ such that for every $f$ with spectrum in $2^j \mathcal O$,
    \begin{align}
        \myR f = 2^{j(n+1-2\alpha)} \phi(2^j \functioninput)* f.
    \end{align}
\end{enumerate}
\end{prop}
\begin{proof}[Proof of Theorem \ref{thm-gen-HKRlem3.3}]
We use    Bony's decomposition,
\begin{align}
    [\myR,v]\theta &= \mathrm I + \mathrm{II } + \mathrm{III} , \\
    \mathrm{I} &= \sum_{q\in \mathbb N}[\myR, S_{q-1} v]\Delta_q \theta  =\sum_{q\in \mathbb N}  \mathrm I_q ,\\
    \mathrm{II} &= \sum_{q\in \mathbb N}[\myR, \Delta_{q-1} v]S_{q-1} \theta= \sum_{q\in \mathbb N} \mathrm{II }_q , \\
    \mathrm{III} &= \sum_{q\ge -1} [\myR, \Delta_q v]\widetilde \Delta_q \theta=  \sum_{q\ge -1} \mathrm{III}_q,
\end{align}
where $\widetilde\Delta_q := \Delta_{q-1} + \Delta_q + \Delta_{q+1}$ (and $\Delta_{-2} := 0$). 
The terms $\mathrm{I}$ and $\mathrm{II}$ are low-high and high-low interactions; the term $\mathrm{III}$ are the interactions at similar frequencies (low-low and high-high). 
\subsubsection*{Estimation of $\mathrm I$}
We have
\begin{align}
   I_q =  2^{(1-2\alpha)q} \big(h_q * (S_{q-1} v \Delta_q \theta) - (S_{q-1} v) ( h_q * \Delta_q \theta)\big),
\end{align}
for some $h_q = 2^{dq} h(2^q x)$ coming from the third part of Proposition \ref{prop-wx-4-1}%
.
By Lemma \ref{lem-hkr3-2},
\begin{align}
\|\mathrm I_q\|_{L^2} 
&\lesssim     2^{(1-2\alpha)q} \|xh_q\|_{L^1} \|\nabla S_{q-1} u\|_{L^2} \|\Delta_q \theta\|_{L^\infty}  \\
&\lesssim 2^{-(2\alpha-1)q} 2^{-q} \|\nabla u\|_{L^2} \|\Delta_q \theta\|_{L^\infty}  \\
&= 2^{-2\alpha q} \|\nabla u\|_{L^2} \|\Delta_q \theta\|_{L^\infty}. 
\end{align}
Therefore, we have (there are no low frequency terms in $\mathrm I$)
\begin{align}
    \|{\mathrm I}\|_{H^s}^2 &\sim \sum_{q\ge 0} 2^{2q s}   \|\mathrm I_q\|_{L^2}^2\\ 
    &\lesssim \|\nabla u\|_{L^2}^2 \sum_{q\ge 0} 2^{2q(s-2\alpha)} \|\Delta_q \theta \|_{L^\infty}^2 \\
    &\le \|\nabla u\|^2_{L^2} \|\theta\|^2_{B^{s-2\alpha}_{\infty,2}}.
\end{align}
\subsubsection*{Estimation of $\mathrm{II}$}
Similarly to $\mathrm{I}$, we write
\begin{align}
    \mathrm{II}_q = 2^{(1-2\alpha)q}\big( h_q * (\Delta _{q} u S_{q-1} \theta) - (\Delta_q v) (h_q * S_{q-1} \theta)\big)
\end{align}
Lemma \ref{lem-hkr3-2} this time gives
\begin{align}
\|\mathrm{II}_q\|_{L^2} 
&\lesssim     2^{(1-2\alpha)q} \|xh_q\|_{L^1} \|\nabla \Delta_q u\|_{L^2} \|S_{q-1} \theta\|_{L^\infty}  \\
&\lesssim  2^{-2\alpha q}  \|\nabla u\|_{L^2} \sum_{j\le q-2}\|\Delta_j \theta\|_{L^\infty}.  
\end{align}
We want to use $ \|\mathrm {II}\|_{H^s}^2 \sim \sum_{q\ge 0} 2^{2q s} \|\mathrm{II}_q \|_{L^2}^2$ again; this time we will use the discrete Young's inequality. 
Multiplying by $2^{sq}$, we note that
\begin{align}
    2^{sq}\|\mathrm{II}_q\|_{L^2} 
    &= 2^{-(2\alpha - s)q} \|\nabla u\|_{L^2} \sum_{j\le q-2}\|\Delta_j \theta\|_{L^\infty}\\
    &\lesssim \|\nabla u\|_{L^2}   \left( 2^{-(2\alpha - s)\functioninput} \star  (2^{-(2\alpha - s)\functioninput}  \|\Delta_\functioninput \theta\|_{L^\infty}) \right) (q),
\end{align}
where $\star$ denotes the discrete convolution on $\mathbb Z_{\ge -1}$,
\[ a \star b (q) = \sum_{q_1+q_2 = q} a_{q_1} b_{q_2}.\]
Applying the discrete Young's inequality (note that $2^{-(2\alpha - s)\functioninput}\in \ell^1  \iff s < 2\alpha$) with parameters $1+\frac12 = \frac11 + \frac12$,
\begin{align}
    \|\mathrm{II}\|_{H^s} 
    &\lesssim \|\nabla u\|_{L^2} \Big\|  2^{-(2\alpha - s)\functioninput} \star  2^{-(2\alpha - s)\functioninput}  \|\Delta_\functioninput \theta\|_{L^\infty} (q)\Big\|_{\ell^2(dq)}
    \\
    &\lesssim_{2\alpha-s} \|\nabla u\|_{L^2} \|\theta\|_{B^{-(2\alpha - s)}_{\infty,2}}. 
\end{align}
\subsubsection*{Estimation of $\mathrm{III}$}
We further split $\mathrm{III}$ into high-high and low-low interactions,
\begin{align}
    \mathrm{III}&=J_1+J_2,\\
    J_1 &= \sum_{q\ge 1} [\myR,\Delta_q v]\widetilde \Delta_q \theta ,\\
    J_2 &= \sum_{q\le 0} [\myR,\Delta_q v]\widetilde \Delta_q \theta .
\end{align}
$J_1$ with no low frequency terms is dealt with as before, giving
\begin{align}
   q\ge 1 \implies  \| [\myR,\Delta_q v]\widetilde \Delta_q \theta \|_{L^2} \lesssim 2^{-2\alpha q}\|\nabla v\|_{L^2} \| \widetilde \Delta_q \theta \|_{L^\infty },
\end{align}
so that 
\begin{align}
    \|\Delta_j J_1\|_{L^2} &\lesssim \|\nabla v \|_{L^2} \sum_{q:q\ge j-4} 2^{-2\alpha q}  \|\widetilde\Delta_q \theta\|_{L^\infty},\\
    2^{js} \|\Delta_j J_1\|_{L^2} &\lesssim \|\nabla v \|_{L^2} \sum_{q\ge j-4} 2^{js} 2^{-2\alpha q}  \|\widetilde\Delta_q \theta\|_{L^\infty} \\
    &= \|\nabla v \|_{L^2} \sum_{q\ge j-4} 2^{(j-q)s} 2^{-(2\alpha-s) q}  \|\widetilde\Delta_q \theta\|_{L^\infty}.
\end{align}
Under the assumption that $s>0$, we have by the discrete Young's inequality again,
\begin{align}
    \|J_1\|_{H^s} \lesssim \|\nabla v\|_{L^2} \|\theta\|_{B^{-(2\alpha - s)}_{\infty,2}}.
\end{align}
For $J_2$, the lowest frequency terms don't have a corresponding annulus for us to apply Lemma \ref{lem-hkr3-2}, so the Riesz transform has to be dealt with in a different way. We proceed without using the commutator structure, instead relying on Bernstein inequalities and $L^2\to L^p$ boundedness of $\myR$ for  $p\in(2,\infty)$ defined by $\frac1p = \frac12 - \frac{2\alpha - 1}d$ (as in Proposition \ref{prop-wx-4-1}): 
\begin{align}
 \| [\myR,\Delta_q v]\widetilde \Delta_q \theta \|_{L^2}
 &\lesssim \|\Delta_q v\|_{L^2} \big(\|\Delta_q \theta\|_{L^\infty} + \| \myR \Delta_q \theta\|_{L^\infty}  \big)
 \\
 &= \|\Delta_q v\|_{L^2} \big(\|\Delta_q \theta\|_{L^\infty} + \| \Delta_q \myR \theta\|_{L^\infty}  \big)
\\
&\lesssim  \|v\|_{L^2} (\|\theta\|_{L^2} + \|\myR\theta\|_{L^p})
\\
&\lesssim \|v\|_{L^2} \|\theta\|_{L^2}.
\end{align}
(Note that all constants from Bernstein inequalities are left implicit because $J_2$ is a finite sum.) This completes the estimation of $\mathrm{III}$ and the Theorem is proved.
\end{proof}

\section{Basic A Priori Estimates}
\label{sectionBasic}
Here we collect some estimates for smooth solutions of \eqref{the-eqn}. Since $\theta$ solves a transport equation with no diffusion, and by testing the $u$ equation with itself, we obtain
\begin{align}
    \|\theta(t)\|_{L^p} &\leq \|\theta_0\|_{L^p}, \qquad p\in[1,\infty],\\
    \frac12\dv{}{t}\|u(t)\|_{L^2}^2 + \|\Lambda^\alpha u(t) \|_{L^2}^2 &\le \|\theta_0\|_{L^2} \|u\|_{L^2}.
\end{align}
These imply
\begin{align}
     \dv{}{t}\|u(t)\|_{L^2} &\le \|\theta_0\|_{L^2},\\
     \frac12\dv{}{t}\|u(t)\|_{L^2}^2 + \|\Lambda^\alpha u(t) \|_{L^2}^2 &\le \|\theta_0\|_{L^2} (\|u_0\|_{L^2}+ \|\theta_0\|_{L^2} t),
\end{align}
and therefore:
\begin{prop}\label{prop-basic-estimates}
Suppose $(u,\theta)$ are smooth solutions of \eqref{the-eqn} with initial data $u_0\in L^2$     and $\theta_0\in L^2\cap L^p$. Then
\begin{align}
    \|\theta(t)\|_{L^p} &\le  \|\theta_0\|_{L^p}, \\
    \|u(t)\|_{L^2}^2 
    +\int_0^t \|\Lambda^\alpha u(s)\|_{L^2}^2 \dd s &\lesssim_{u_0,\theta_0} 1+t^2. 
\end{align}
\end{prop}

\begin{prop}\label{prop-basic-gamma-est}
Suppose $(u,\theta)$ are smooth solutions of \eqref{the-eqn} with initial data $u_0 \in H^1$ and $\theta \in L^2\cap L^\infty$. Then for $\omega = \nabla^\perp \cdot  u$,
\begin{align}
    \|\omega(t)\|_{L^2}^2 + \int_0^t \|\omega(s) - \myR\theta(s)\|_{\dot H^{\alpha}}^2 \dd s \le \Phi_1(t).
\end{align}    
\end{prop}
\begin{proof}
As mentioned, we set $\Gamma = \omega - \myR\theta$. It solves the equation
\begin{align}
        (\partial_t + u\cdot\nabla + \Lambda^{2\alpha} )\Gamma = [\myR,u\cdot\nabla ]\theta   
\end{align}

Taking the $L^2$ inner product with $\Gamma$ and using that $\nabla\cdot u=0$ in the form of the identity $[\myR, u \cdot \nabla] \theta=\nabla\cdot ([\myR, u] \theta)$,
\begin{align}
    \frac12 \frac{\dd}{\dd t} \|\Gamma\|_{L^2}^2 + \|\Gamma \|_{\dot H^{\alpha}}^2 
    &= \int_{\mathbb R^2} \nabla\cdot ([\myR, v] \theta) \Gamma
    \\ & \le \| [\myR, v]\theta \|_{\dot H^{1-\alpha}} \| \Gamma\|_{\dot H^{\alpha}}.
\end{align}
By Theorem \ref{thm-gen-HKRlem3.3}, Proposition \ref{prop-basic-estimates}, and Proposition \ref{prop-wx-4-1},
\begin{align}
        \|[\myR,u]\theta\|_{\dot H^{1-\alpha}} 
        &\lesssim \|\nabla u\|_{L^2} \|\theta\|_{B^{1-3\alpha}_{\infty,2}} + \|u\|_{L^2} \|\theta\|_{L^2} 
        \\
        &\lesssim \|\omega\|_{L^2}\|\theta\|_{L^{\infty}} + \|u\|_{L^2} \|\theta\|_{L^2}
        \\
        &\lesssim \|\omega\|_{L^2} + 1 +t
        \\
        &\lesssim \|\Gamma\|_{L^2} + \|\myR \theta\|_{L^2} + 1 + t
        \\
        &\lesssim \|\Gamma\|_{L^2} + 1 + t.
\end{align}
Therefore, 
\begin{align}
    \frac12 \frac{\dd}{\dd t} \|\Gamma\|_{L^2}^2 + \|\Gamma \|_{\dot H^{\alpha}}^2 = \int_{\mathbb R^2} \nabla\cdot ([\myR, v] \theta) \Gamma \lesssim (\|\Gamma\|_{L^2} + 1 + t)\| \Gamma\|_{\dot H^{\alpha}}.
\end{align}
Young's inequality for products gives
\begin{align}
\frac{\dd}{\dd t} \|\Gamma(t)\|_{L^2}^2 + \|\Gamma \|_{\dot H^{\alpha}}^2 \lesssim \|\Gamma\|_{L^2}^2 + 1 + t^2.
\end{align}
An application of Grownwall's inequality finishes the proof.
\end{proof}

\section{A Priori Estimates for the Vorticity}  
\label{sectionAPrioriVort}
\begin{prop}\label{thm-omega-lp-estimate}
Let $(u,\theta)$ be a smooth solution of \eqref{the-eqn}, and let $u_0 \in H^1\cap \dot W^{1,p}$, with $p\in(2,\infty)$,  $\nabla\cdot u_0=0$, and $\theta_0 \in L^2\cap L^\infty$. Then for $\omega = \nabla^\perp \cdot  u$,
\begin{align}
    \|\omega(t)\|_{L^p} \le \Phi_1(t).
\end{align}    
\end{prop}
\begin{proof}
    We will again use the equation for $\Gamma$,
    \[ \partial_t \Gamma + u\cdot \nabla \Gamma + \Lambda^{2\alpha}\Gamma = [\myR, u\cdot \nabla ]\theta \]
    From the $L^p$ estimates in Corollary 3.6 of Cordoba--Cordoba (they prove the estimate for the homogeneous equation $\Gamma_t + u\cdot \nabla \Gamma + \Lambda^{2\alpha} \Gamma = 0$; what we need follows by Duhamel's principle),
    \[ \|\Gamma(t)\|_{L^p} \le \|\Gamma_0\|_{L^p} + \int_0^t \|[\myR, u\cdot \nabla ]\theta\|_{L^p} \dd s. \]
 By Proposition \ref{wx-prop42}, 
 \[
\left\|\left[\myR, u \cdot \nabla\right] \theta\right\|_{B_{p, 1}^{0}} \lesssim \|\nabla u\|_{L^p}\left(\|\theta\|_{B_{\infty, 1}^{1-2\alpha }}+\|\theta\|_{L^{2}}\right) \lesssim \|\omega\|_{L^p}(\|\theta_0\|_{L^2}+\|\theta_0\|_{L^\infty}).
\]
Therefore, using $\Gamma = \omega - \myR \theta$ and the boundedness of $\myR$,
\[ \|\omega(t)\|_{L^p} \lesssim_{\theta_0,\omega_0} 1+ \int_0^t \|\omega(s)\|_{L^p} \dd s, \]
and Gronwall's inequality completes the proof.
\end{proof}

The following result is the analogue of Theorem 6.3 in \cite{hmidi2009global}.
\begin{thm}\label{hkr-6.3}
    Let $(u,\theta)$ solve \eqref{the-eqn} with initial data $u_0 \in H^1 \cap \dot W^{1,p}$, $\nabla\cdot u=0$, and $\theta_0 \in L^2\cap L^\infty$ with $p\in(2,\infty)$.  Then for $r \in [2,p]$ and $\rho\in[1,r /2)$,
\begin{align}
    \|\omega - \myR \theta\|_{\widetilde L^\rho_t B^{4\alpha /r}_{r,1}} \le \Phi_1(t).
\end{align}
\end{thm}

\begin{proof}
    
For \(q \in \mathbb{N}\) we set \(\Gamma_{q}=\Delta_{q} \Gamma .\) Then, we localise in frequencies the equation (5.1) for \(\Gamma\) to get
\[
\begin{aligned}
\partial_{t} \Gamma_{q}+u\cdot \nabla \Gamma_{q}+\Lambda^{2\alpha} \Gamma_{q} 
&=-\left[\Delta_{q}, u \cdot \nabla\right] \Gamma+\Delta_{q}([\myR, u \cdot \nabla] \theta) \\
&:=f_{q}.
\end{aligned}
\]
Multiplying the above equation by \(\left|\Gamma_{q}\right|^{r-2} \Gamma_{q}\) and integrating in the space variable we find
\[
\frac{1}{r} \dv{}t\left\|\Gamma_{q}(t)\right\|_{L^{r}}^{r}+\int_{\mathbb{R}^{2}}\left(\Lambda^{2\alpha} \Gamma_{q}\right)\left|\Gamma_{q}\right|^{r-2} \Gamma_{q} \dd x \leq\left\|\Gamma_{q}(t)\right\|_{L^{r}}^{r-1}\left\|f_{q}(t)\right\|_{L^{r}}.
\]
From \cite{chen2007newbernstein} or \cite{li2013frequency}, we have the following generalised Bernstein inequality,
\[
\int_{\mathbb{R}^{d}}\left(\Lambda^{2\alpha} \Gamma_{q}\right)\left|\Gamma_{q}\right|^{r-2} \Gamma_{q} \dd x \geq c 2^{2\alpha q}\left\|\Gamma_{q}\right\|_{L^{r}}^{r},
\]
valid for some \(c>0\) independent of \(q\), and hence we find
\[
\frac{1}{r} \dv{}t\left\|\Gamma_{q}(t)\right\|_{L^{r}}^{r}+c 2^{2\alpha q}\|\Gamma_q(t)\|_{L^{r}}^{r} \leq\left\|\Gamma_{q}(t)\right\|_{L^{r}}^{r-1}\left\|f_{q}(t)\right\|_{L^{r}}.
\]
This yields (writing $\Gamma_0:= \omega_0 - \myR \theta_0 $ and $(\Gamma_0)_{q}:=\Delta_q \Gamma_0$) 
\[
\left\|\Gamma_{q}(t)\right\|_{L^{r}} \leq e^{-c t 2^{2\alpha q}}\left\|(\Gamma_0)_{q}\right\|_{L^{r}}+\int_{0}^{t} e^{-c(t-\tau) 2^{2\alpha q}}\left\|f_{q}(\tau)\right\|_{L^{r}} \dd \tau.
\]
By taking the \(L^{\rho}[0, t]\) norm and by using convolution inequalities, we find (since $f_q = -\left[\Delta_{q}, u \cdot \nabla\right] \Gamma+\Delta_{q}([\myR, u \cdot \nabla] \theta) $, and $\rho\in[1,p/2]$)
\begin{align}
    &2^{2\alpha q \frac{2}{r}}\left\|\Gamma_{q}\right\|_{L_{t}^{\rho} L^{r}} 
   \\ &\lesssim_p 2^{2\alpha q\left(\frac{2}{r}-\frac{1}{\rho}\right)}\left\|(\Gamma_0)_{q}\right\|_{L^{r}}+2^{2\alpha q\left(\frac{2}{r}-\frac{1}{\rho}\right)} 
    \int_{0}^{t}\left\|\left[\Delta_{q}, u \cdot \nabla\right] \Gamma(\tau)\right\|_{L^{r}} \dd \tau
    \\
    &\quad +2^{2\alpha q\left(\frac{2}{r}-\frac{1}{\rho}\right)} \int_{0}^{t}\left\|\Delta_{q}([\myR, u \cdot \nabla] \theta)(\tau)\right\|_{L^{r}} \dd\tau \label{Gammaq-need-to-sum}.
\end{align}
To estimate the second integral of the RHS of \eqref{Gammaq-need-to-sum}, we use 
Proposition \ref{wx-prop42}.
This gives uniformly in $q\ge 0$, 
\begin{align}
    \left\|\Delta_q \left[\myR, u \cdot \nabla\right] \theta\right\|_{L^{r}}    
    &\le \| \left[\myR, u \cdot \nabla\right] \theta \|_{B^0_{r,\infty}} 
    \\
    &\leq \|\nabla u\|_{L^r} (\|\theta\|_{B_{\infty,\infty}^{1-2\alpha }} + \|\theta\|_{L^2})
    \\
    &\le \Phi_1(t). 
\end{align}
For the first integral of the RHS of  \eqref{Gammaq-need-to-sum} we use the following lemma:

\begin{lem}[Lemma 6.4 of \cite{hmidi2009global}]\label{hkr-6.4} Let \(v\) be a smooth divergence-free vector field and \(f\) be a smooth scalar function. Then, for all \(a \in[1, \infty]\) and \(q \geq-1\),
\[
\left\|\left[\Delta_{q}, v \cdot \nabla\right] f\right\|_{L^{a}} \lesssim\|\nabla v\|_{L^{r}}\|f\|_{B_{a, 1}^{2/a}}.
\]
\end{lem}
Using Lemma \ref{hkr-6.4}, the $L^2$ bound (Proposition \ref{prop-basic-gamma-est}), and the $L^p$ bound on $\omega$ (Proposition \ref{thm-omega-lp-estimate}), we can interpolate (recall $r\in[2,p]$) to bound $\|\nabla v\|_{L^r} \le \Phi_1(t)$, giving
\begin{align}
    \int_0^t \|[\Delta_q ,u\cdot \nabla ] \Gamma\|_{L^r}\dd \tau  \le \|\nabla u\|_{L^r} \|\Gamma\|_{B_{r, 1}^{2/r}} \le \Phi_1(t)  \|\Gamma\|_{L^1 B_{r, 1}^{2/r}} 
\end{align}
Now we show how to estimate the sum in $q\ge-1$ of \eqref{Gammaq-need-to-sum}. For high frequencies, since $r>2\rho$ i.e. $\frac2r - \frac1\rho < 0$, we have
\begin{align}
    \sum_{q\ge N} 2^{2\alpha q{ (\frac 2r - \frac 1\rho )}} \int_0^t \| [\Delta_q , u\cdot \nabla ] \Gamma \|_{L^r} \dd\tau \le 2^{-\alpha N(\frac 1\rho -\frac 2r )} \|\Gamma\|_{L^1 B^{2/ r}_{r,1}} ,
\end{align}
so that by \eqref{Gammaq-need-to-sum}, for an $N$ to be chosen,
\begin{align}
    &\sum_{q\ge N}  2^{2\alpha q \frac{2}{r}}\left\|\Gamma_{q}\right\|_{L_{t}^{\rho} L^{r}} 
    \\ &\le \underbrace{\|\Gamma_0\|_{B^{2\alpha ( \frac2r - \frac1\rho)}_{r,1}} }_{ \le \|\Gamma_0\|_{L^r}}
    + 2^{-\alpha N ( \frac1\rho - \frac2 r)} \Phi_1(t) (1+ \|\Gamma\|_{L^1_t B^{2/r}_{r,1}}).
\end{align}
(The $1$ and $\|\Gamma\|_{L^1_tB^{2/r}_{r,1}}$ is from the first and second integral terms, respectively.) For low frequencies of $\Gamma$, we just use
\begin{align}
    \sum_{q<N} 2^{2\alpha q \frac2r} \|\Gamma_q\|_{L^\rho_tL^r_x} \le 2^{\frac2r 2\alpha N} \|\Gamma\|_{L^\rho_t L^r_x}  \le 2^{\frac2r 2\alpha N} \|\Gamma\|_{L^\infty_t L^r_x} t^{1/\rho} \le 2^{\frac2r 2\alpha N}  \Phi_1(t).
\end{align}
replacing $\|\Gamma\|_{B^{2/r}_{r,1}}$ on the right by the worse term $\|\Gamma\|_{B^{4\alpha /r}_{r,1}}$, we see that
\begin{align}
    \|\Gamma\|_{\widetilde L^\rho B^{4\alpha/r}_{r,1}} 
     &\le  2^{-\alpha N ( \frac1\rho - \frac2 r)} \Phi_1(t) (1+ \|\Gamma\|_{L^1_t B^{4\alpha /r}_{r,1}}) + 2^{\frac2r 2\alpha N}  \Phi_1(t)
    \\&\le  2^{-\alpha N ( \frac1\rho - \frac2 r)} \Phi_1(t) (1+ \|\Gamma\|_{\widetilde L^\rho_t B^{4\alpha /r}_{r,1}}) + 2^{\frac2r 2\alpha N}  \Phi_1(t).
\end{align}
Taking $N\gg 1$, we obtain $\|\Gamma\|_{\widetilde L^\rho B^{4\alpha/r}_{r,1}}  \le \Phi_1(t)$ as claimed.
\end{proof}
\begin{prop}\label{prop-hkr66} Let $(u,\theta)$ be a smooth solution of \eqref{the-eqn}, and let $u_0 \in H^1\cap \dot W^{1,p}$, with $p\in(2,\infty)$, $\nabla\cdot u_0 = 0$, and $\theta_0 \in L^2\cap L^\infty$. Then we have
\begin{align}
    \|u\|_{\widetilde L_t^\rho B^{\frac2r(2\alpha-1)+1}_{\infty,1} }
      \le \Phi_1(t),
\end{align}
for every $\rho \in [1,2\wedge \frac p2)$ and  $r>2$. 

\end{prop}
\begin{proof}The result is stronger for $r$ closer to 2, so without loss of generality, make $r$ smaller so that $r\le p$ (this is for the application of Theorem \ref{hkr-6.3}). By triangle inequality and $\omega = \Gamma +\myR\theta$,
\begin{align}
    \|\omega\|_{\widetilde L_t^\rho B^{\frac2r(2\alpha-1)}_{\infty,1}} \le \|\Gamma\|_{\widetilde L_t^\rho B^{\frac2r(2\alpha-1)}_{\infty,1}} + \|\myR\theta\|_{\widetilde L_t^\rho B^{\frac2r(2\alpha-1)}_{\infty,1}}.
\end{align}
The first term is controlled by some $\Phi_1$ by the embedding $\widetilde L^\rho B^{4\alpha/r}_{r,1} \hookrightarrow \widetilde L^\rho B^{\frac2r(2\alpha-1)}_{\infty,1}     $ and the previous result, Proposition \ref{hkr-6.4}. Recalling that $\myR = \Lambda^{1-2\alpha } \mathcal R$, we can use Bernstein inequalities to see that
\begin{align}
    \|\myR\theta\|_{\widetilde L^\rho B^{\frac2r(2\alpha-1)}_{\infty,1}} 
    &= \sum_{q\ge-1} 2^{\frac2r(2\alpha-1)q}\|\Delta_q \myR \theta\|_{L^\rho _t L^\infty}
    \\
    &\lesssim \|\myR \Delta_{-1}\theta\|_{L^\rho_tL^\infty} + \sum_{q\ge 0} 2^{\frac2r(2\alpha-1)q-(2\alpha-1)q}\|\Delta_q\theta\|_{L^\rho_tL^\infty}
    \\
    &\lesssim \| \Delta_{-1}\theta\|_{L^\rho_tL^2} + \sum_{q\ge 0} 2^{(\frac2r-1)(2\alpha-1)q}\|\theta\|_{L^\rho _tL^\infty }
    \\
    &\le t^{1/\rho} \|\theta_0\|_{L^2} + t^{1/\rho}\|\theta_0\|_{L^\infty}\sum_{q\ge0}2^{(\frac 2r-1)(2\alpha-1)q},
\end{align}
where we have used that $\|\theta(t)\|_{L^p}\le\|\theta_0\|_{L^p}$. Since $r>2$,
\begin{align}
 \sum_{q\ge0}2^{(\frac 2r-1)(2\alpha-1)q} < \infty.
\end{align}
Hence, $\|\myR \theta\|_{\widetilde L^\rho B^{\frac 2r (2\alpha - 1) }_{\infty,1}} \lesssim_{\theta_0} t^{1/\rho}$, and therefore $\|\omega\|_{ \widetilde L_t^\rho  B^{\frac2r(2\alpha-1)}_{\infty,1} } \le \Phi_1(t). $
Then we obtain (using Bernstein's inequality for the low frequency term) 
\begin{align}
    \|u\|_{\widetilde L^\rho B^{\frac2r(2\alpha-1)+1}_{\infty,1}} &\lesssim \|\Delta_{-1} u\|_{\widetilde L^\rho  L^\infty} + \|\omega\|_{\widetilde L^\rho B^{\frac2r(2\alpha-1)}_{\infty,1}}
    \\
   & \lesssim 1+t    + \|\omega\|_{\widetilde L^\rho  B^{\frac2r(2\alpha-1)}_{\infty,1}}\\
   & \le \Phi_1(t).
\end{align}
 \end{proof}

\section{Uniqueness of solutions}
\label{sectionUniqueness}
In order to prove the uniqueness, we will rely on the following Osgood lemma, which can be found as Theorem 5.2.1 in \cite{chemin1998perfect}, or in \cite{elgindi2014osgood}:

\begin{lem}[Osgood Lemma]\label{lem-osgood} Let $\gamma \in L_{\mathrm{loc}}^{1}\left(\mathbb{R}_{+} ; \mathbb{R}_{+}\right), \mu$ a continuous non-decreasing function, $a \in \mathbb{R}_{+}$ and $\alpha$ a measurable function satisfying
\[
0 \leq \alpha(t) \leq a+\int_{0}^{t} \gamma(\tau) \mu(\alpha(\tau)) \dd \tau, \quad \forall t \in \mathbb{R}_{+}
\]
If we assume that $a>0$ then
\[
-\mathcal{M}(\alpha(t))+\mathcal{M}(a) \leq \int_{0}^{t} \gamma(\tau) \dd\tau \quad \text { with } \quad \mathcal{M}(x):=\int_{x}^{1} \frac{\dd r}{\mu(r)}.
\]
If we assume $a=0$ and $\lim _{x \rightarrow 0^{+}} \mathcal{M}(x)=+\infty,$ then $\alpha(t)=0, \forall t \in \mathbb{R}_{+}$.
\end{lem}
\begin{rem}\label{rem-explicit-bd}
In the case $\mu(x) = x(1-\log x)$, an explicit bound is also given in \cite{chemin1998perfect}. We will use the similar function $\mu(x):= x \log(e+1/x)$, for which in the case $a<1/e$, we have the estimate
\begin{align}
    \alpha(t) \le \exp\left[-\exp\left (e-1+\mathcal M(a) - \int_0^t \gamma \right)\right]. \label{eqn-explicit-bd}
\end{align}
This estimate follows elementarily from $\mu(r) \ge -r \log r$ for $0<r<1/e$. In particular, note that $\alpha(t) \to 0$ as $a\to 0$. 
\end{rem} We now recall some  results (with mild modifications) from Section 4 of \HKR's paper \cite{hmidi2009global}. 
\subsection{Estimates for transport-diffusion models}
First, we have a result for transported scalars, which we will apply to the temperature $\theta$ of our system.
\begin{prop}\label{prop-linear-scalar-estimate}
    Let $v$ be a smooth divergence-free vector field. Then every scalar solution $\psi$ of the equation
    \[ \partial_t \psi + v \cdot \nabla \psi = f,\quad  \psi|_{t=0} = \psi_0,\]
    satisfies for every $p\in[1,\infty]$,
    \begin{align}
        \|\psi(t)\|_{B^{-1}_{p,\infty}} \le Ce^{ C \int_0^t \| v(\tau)\|_{B^1_{\infty,1}} \dd\tau  } \left( \|\psi_0\|_{B^{-1}_{p,\infty}} + \int_0^t \|f(\tau)\|_{B^{-1}_{p,\infty}} \dd\tau \right) .
    \end{align}
\end{prop}
Secondly, we have the following proposition for the linearised velocity equation. The proof is similar to the proof in \HKR, so we omit the details. We will only apply this proposition with $s=0$.
\begin{prop}\label{prop-linear-velocity-eqn}
    Let $v$ be a smooth divergence-free vector field, $s\in(-1,1)$, and $\rho \in [1,\infty]$. Let $u$ be a smooth solution of the linear system
\[ \partial_t u + v\cdot \nabla u + \Lambda^{2\alpha} u + \nabla p = f,\quad  \nabla \cdot u = 0.\]
Then we have for each $t\in [0,\infty)$, with $s':= s -2\alpha(1-\frac1\rho)$,
\[\|u\|_{L^\infty_t B^s_{2,\infty}}\!\! \le Ce^{C \int_0^t \|\nabla v(\tau)\|_{L^\infty} \dd \tau }\!\left( \|u_0\|_{B^s_{2,\infty}}\!\! + (1+t^{1-1/\rho})\|f\|_{\widetilde L^\rho_t B^{s'}_{2,\infty}} \right).\]
\end{prop}

We also need the following estimates adapted from the Appendix of \cite{hmidi2009global}; essentially the same proofs can prove these lemmas, so we will omit them.
\begin{lem}\label{lem-610-2} For every $s\in[-1,0]$, if $v$ is a smooth divergence-free vector field, and $f$ is a smooth function, then
\[ \|v\cdot \nabla f\|_{B^s_{2,\infty}} \lesssim \|f\|_{B^{1+s}_{\infty,\infty}} \|v\|_{B^0_{2,1}} .   \]
\end{lem}
\begin{lem}\label{lem-log-interpol}
    If $v\in H^1$, then
$    \|v\|_{B^0_{2,1}}  \lesssim \|v\|_{B^0_{2,\infty}} \log\left(e + \frac{\|v\|_{H^1}}{\|v\|_{B^0_{2,\infty}} } \right) $.
\end{lem}
Now we prove the uniqueness of solutions.
\begin{thm}\label{thm-uniqueness}
    Let $\alpha\in(\frac12,1)$. Then the equation \eqref{the-eqn} can have at most one solution pair $(u,\theta)$ (in the sense of Definition \ref{defn-solution}) in the space
    \begin{align}
        \mathcal X_t := ( L^1_t B^1_{\infty,1}\cap L^\infty_t H^1) \times (L^\infty_t B^{-1}_{2,\infty}\cap L_t^1 L^\infty   ).
    \end{align}
\end{thm}
\begin{proof}
    Let $(u^1,\theta^1)$ and $(u^2,\theta^2)$ be two solutions to \eqref{the-eqn} in the space $\mathcal X_T$. Set $u = u^1 - u^2$ and $\theta = \theta^1 - \theta^2$. Then, they solve
    \begin{align}
        \partial_t u + u^2 \cdot \nabla u + \Lambda^{2\alpha} u + \nabla p &= - u \cdot \nabla u^1 + \theta e_2,
        \\
        \partial_t \theta + u^2 \cdot \nabla \theta &= - u \cdot \nabla \theta^1 ,
    \end{align} 
with zero initial data. For the velocity $u$, we write $u=V_1+V_2$ where $V_i$ solve the equations
\begin{align}
    \partial_t V_1 + u^2 \cdot \nabla V_1 + \Lambda^{2\alpha} V_1 + \nabla p_i &= -u\cdot\nabla u^1,
    \\
    \partial_t V_2 + u^2 \cdot \nabla V_2 + \Lambda^{2\alpha} V_2 + \nabla p_i &= \theta e_2,
\end{align} 
Proposition \ref{prop-linear-velocity-eqn}, first with $\rho=1$ and then with $\rho=\infty$, gives
\begin{align}
    \|V_1\|_{L^\infty_t B^0_{2,\infty} }&\le Ce^{C \int_0^t \| \nabla u^2 (\tau)\|_{L^\infty} \dd \tau } \left( \|u_0\|_{B^0_{2,\infty}} + \|u\cdot \nabla u^1\|_{L^1_tB^0_{2,\infty}} \right) ,
    \\
        \|V_2\|_{L^\infty_t B^0_{2,\infty} }&\le Ce^{C \int_0^t \| \nabla u^2 (\tau)\|_{L^\infty} \dd \tau } \left( \|u_0\|_{B^0_{2,\infty}} + (1+t) \|\theta\|_{L^\infty_t B^{-1}_{2,\infty}} \right) ,
\end{align}
since $ L^\infty_tB^{-1}_{2,\infty}\hookrightarrow L^\infty_tB^{-2\alpha}_{2 ,\infty} =  \widetilde L^\infty_tB^{-2\alpha}_{2 ,\infty} $. Together, we obtain
\begin{align}
    \|u\|_{L^\infty_t B^0_{2,\infty}}
    &\leq  Ce^{C \int_0^t \| \nabla u^2 (\tau)\|_{L^\infty} \dd \tau } \\ &  \times  \left( \|u_0\|_{B^0_{2,\infty}} + \|u\cdot \nabla u^1\|_{L^1_tB^0_{2,\infty}} + (1+t) \|\theta\|_{L^\infty_t B^{-1}_{2,\infty}}  \right). \quad \label{eqn-the-thing-to-bound}
\end{align}
By Lemma \ref{lem-610-2} with $s=0$, we have 
\begin{align}\label{eqn-application-prop63}
    \|u\cdot \nabla u^1\|_{B^0_{2,\infty}} \lesssim \|u^1\|_{B^1_{\infty,\infty}} \|u\|_{B^0_{2,1}}.
\end{align} 
By using the  logarithmic interpolation inequality of Lemma \ref{lem-log-interpol}, we obtain
\begin{align}
    \|u\|_{B^0_{2,1}} &\lesssim \|u\|_{B^0_{2,\infty}} \log\left(e + \frac{1}{\|u\|_{B^0_{2,\infty}} } \right) \log \left(e + \|u\|_{H^1}  \right) .
\end{align}
We have used   $\log(e+ab) \le \log(e+a)\log(e+b)$ which can be proven for instance by using Bernoulli's inequality for  $a,b\ge 0$ in the form $(1+(1+a)b) \le (1+a)^{1+b}$. Writing $\mu(a) := a \log(e+1/a)$, the above inequality can be rewritten as
\begin{align}\label{eqn-log-interpol}
    \|u\|_{B^0_{2,1}} \lesssim   \mu(\|u\|_{B^0_{2,\infty}})\log \left(e + \|u\|_{H^1}  \right) .
\end{align} 
For the temperature, Proposition \ref{prop-linear-scalar-estimate} with $p=2$ gives
\begin{align}\label{eqn-theta-bound}
        \|\theta\|_{L^\infty_tB^{-1}_{2,\infty}}\!\! \le Ce^{ C  \| u^2\|_{L^1_tB^1_{\infty,1}}   } \!\!\left( \|\theta_0\|_{B^{-1}_{2,\infty}}\!\! + \int_0^t\! \|v\cdot\nabla\theta^1\|_{B^{-1}_{2,\infty}} \dd\tau \right).\quad 
        \end{align}
By Lemma \ref{lem-610-2} with $s=-1$, 
\begin{align}
\label{eqn-application-prop63-2}
        \|u\cdot\nabla\theta^1\|_{B^{-1}_{2,\infty}} \le \|\theta^1\|_{B^0_{\infty,\infty}}\| u\|_{B^0_{2,1}} \le \|\theta_0\|_{L^\infty}\|u\|_{B^0_{2,1}}  . 
\end{align}
By combining \eqref{eqn-the-thing-to-bound} with \eqref{eqn-application-prop63}, \eqref{eqn-log-interpol}, \eqref{eqn-theta-bound}, \eqref{eqn-application-prop63-2}, we obtain
\begin{align}
        &\|u\|_{L^\infty_t B^0_{2,\infty}} +  \|\theta\|_{L^\infty_tB^{-1}_{2,\infty}}\\
        &\le 
        Ce^{ C  \| u^2\|_{L^1_tB^1_{\infty,1}}   } \left( \|\theta_0\|_{B^{-1}_{2,\infty}} + \|\theta_0\|_{L^\infty}\int_0^t  \mu(\|u\|_{B^0_{2,\infty}})\log \left(e + \|u\|_{H^1}  \right) \dd\tau  \right) \\
        & + Ce^{C  \| \nabla u^2 (\tau)\|_{L^1_tL^\infty}  }  \Bigg( \|u_0\|_{B^0_{2,\infty}} \!+ \int_0^t \|u^1\|_{B^1_{\infty,\infty}} \mu(\|u\|_{B^0_{2,\infty}})\log \left(e + \|u\|_{H^1}  \right) \dd\tau  \\
        & + C(1+t) e^{C \|u^2\|_{L^1_t B^1_{\infty,1}}} \left[ \|\theta_0\|_{B^{-1}_{2,\infty}} \!+ \|\theta_0\|_{L^\infty} \int_0^t  \mu(\|u\|_{B^0_{2,\infty}})\log \left(e + \|u\|_{H^1}  \right) \dd\tau  \right]\Bigg) \\
        &\le C(1+t) e^{C \|u^2\|_{L^1_t B^{1}_{\infty,1} }}\Bigg( \|u_0\|_{B^0_{2,\infty}} + \|\theta_0\|_{B^{-1}_{2,\infty} } +\log(e+\|u\|_{L^\infty_t H^1})(1+\|\theta_0\|_{L^\infty})\\
        &   \times \int_0^t (1+\|u^1(\tau)\|_{B^1_{\infty,\infty}}) \mu(\|u\|_{L^\infty_\tau B^0_{2,\infty}} + \|\theta\|_{L^\infty_\tau B^{-1}_{2,\infty}} ) \dd\tau   \Bigg).
\end{align}
Setting
\begin{align}
    X(t) &\coloneq \|u\|_{L^\infty_t B^0_{2,\infty}} +  \|\theta\|_{L^\infty_tB^{-1}_{2,\infty}},\\
    f(t) &\coloneq  C(1+t) e^{C \|u^2\|_{L^1_t B^{1}_{\infty,1} }} + \log(e+\|u\|_{L^\infty_t H^1})(1+\|\theta_0\|_{L^\infty}), \\
    g(t) &\coloneq 1+\|u^1(\tau)\|_{B^1_{\infty,\infty}},
\end{align}
We obtain the following integral inequality for $X$:
\begin{align}
    X(t) \le f(t)\left( X(0) + \int_0^t g(\tau) \mu(X(\tau) )\dd \tau  \right).
\end{align}
Since $X(0)=0$, by Lemma \ref{lem-osgood}, $X(t)=0$ for every $t$, which proves the uniqueness of solutions.
\end{proof}

\section{Existence of solutions}
\label{sectionExistence}
We now use the above results to prove our main technical result.
\begin{proof}[Proof of Theorem \ref{main}] The uniqueness is guaranteed by Theorem \ref{thm-uniqueness} so we focus on the existence. Following \cite{chae_nam_1997} (or Chapter 3 of  \cite{majda2002}), if the initial data are smooth, we can obtain smooth solutions in $H^m$ to \eqref{the-eqn}. By the a priori estimate of Proposition \ref{prop-hkr66} and the Beale--Kato--Majda type blow-up criterion \cite{chae_nam_1997}, the solution is globally defined.

Now let $u_0\in H^1\cap \dot{W}^{1,p}$ with $\nabla\cdot u_0 = 0$, for some $p>2$, and $\theta_0 \in L^1\cap L^\infty$. Consider the sequence of initial data $u_{0,n}, \theta_{0,n}$ defined by Littlewood--Paley projections,
\begin{align}
    u_{0,n} &\coloneq S_n u_0, \\
    \theta_{0,n} &\coloneq S_n \theta_0 .
\end{align}
These functions are smooth, so they define smooth solutions $u_n,\theta_n$. Moreover, we have a priori estimates uniform in $n$ that imply for any $\epsilon>0$ and $\rho\in [1,2\wedge \frac p2)$,
\begin{align}
    \theta_n &\in L^\infty_t (L^1\cap L^\infty), \\
    u_n &\in L^\infty_t W^{1,p} \cap L^\rho_t  B^{2\alpha-\epsilon}_{\infty,1}.
\end{align}
The spaces $B^{2\alpha-\epsilon }_{\infty,1} $ and $C^{2\alpha-\epsilon}=B^{2\alpha-\epsilon}_{\infty,\infty}$ are essentially equivalent by standard embeddings of Besov spaces due to the small loss in $\epsilon$, so we can replace one with the other at any point.

  Up to subsequences, we have that $\theta_n$ and $u_n$ converge weakly to functions $\theta$ and $u$. By taking $(S_n u_0 -S_m u_0, S_n\theta_0 - S_m \theta_0)$ as initial data, the proof of Theorem \ref{thm-uniqueness} and Remark \ref{rem-explicit-bd} gives that as soon as 
  \[ \|S_n u_0 -S_m u_0\|_{ B^0_{2,\infty}} +  \|S_n \theta_0 -S_m \theta_0\|_{B^{-1}_{2,\infty}} \le 1/e, \]
  we obtain
\begin{align}
    &\|u_n -u_m\|_{L^\infty_t B^0_{2,\infty}} +  \|\theta_n - \theta_m\|_{L^\infty_TB^{-1}_{2,\infty}}
    \\ &\le F(t,\|S_n u_0 -S_m u_0\|_{ B^0_{2,\infty}} +  \|S_n \theta_0 -S_m \theta_0\|_{B^{-1}_{2,\infty}} ) \label{eqn-decaying-bd},
\end{align} 
for an explicit function $F$ given by \eqref{eqn-explicit-bd} with $F(t,s)\to 0$ as $s\to 0$. Thus $u_m$ is Cauchy and hence strongly convergent to $u$  in $L^\infty_tB^0_{2,\infty}$. Interpolation yields strong convergence of $u_n$ to $u$ in $L^2([0,t]\times \mathbb R^2)$. This implies $u_n\otimes u_n $ to $u\otimes u$ strongly in $L^1([0,t]\times\mathbb R^2)$. Also, since $\theta_n\rightharpoonup\theta$ in $L^2$, the product $u_n\theta_n\rightharpoonup u\theta$ in $L^1$. Thus, all terms in Definition \ref{defn-solution} make sense, and $(u,\theta)$ is a solution to \eqref{the-eqn}.
\end{proof}
Finally, for completeness, we give a standard argument that shows the boundary regularity for the temperature patches is preserved.
\begin{proof}[Proof of Theorem \ref{cor-Boundary-reg-intro-ver}]
Let $X=X(t,\xi)$ denote the flow of the vector field $u \in L^1_t C^{2\alpha'}$, i.e. the solution of
\begin{align}
     X_t(t,\xi)  &= u(X(t,\xi) , t), \\ X(0,\xi) &= \xi.
\end{align}
Taking the gradient in $\xi$ gives
\begin{align}
     (\nabla_\xi X)_t = \nabla u(X,t) \nabla_\xi X,
\end{align}
so that
\[
    \|\nabla_\xi X(t,\xi)\|_{L^\infty_\xi} \leq  \|\nabla_\alpha X(0,\xi)\|_{L^\infty_\xi} e^{\int_0^t \| \nabla u(s)\|_{L^\infty} \dd s},\\
\]
and similarly,
\[
    [\nabla_\xi X(t,\xi) ]_{C^{2\alpha'-1}} \le \|\nabla_\xi X(0,\xi)\|_{C^{2\alpha'-1}}e^{C\int_0^t \| u(s)\|_{C^{2\alpha'}}  \dd s}.
\]
This shows that the flow remains in $ C^{2\alpha'-1}$. To apply this to the boundary of the patch, suppose that at time $t=0$, $\partial D_0$ is parameterised by the curve $\gamma_0 \in C^{2\alpha'}([0,1];\mathbb R^2)$. Then at time $t$, the flow transports it to the curve
\[\gamma(t,s): = X(t,\gamma_0(s)).\]
Then $\partial_s \gamma = \nabla_\xi X(t,\gamma_0(s)) \gamma_0'$, and the \Holder seminorm of $\partial_s \gamma $ is controlled:
\[
    [\partial_s \gamma  ]_{C^{2\alpha'-1}} \le [\nabla_\xi X(t,\gamma_0)]_{C^{2\alpha'-1}} \|\gamma_0'\|_{L^\infty}^2 + \|\nabla_\xi X\|_{L^\infty}[\gamma_0']_{C^{2\alpha'-1}}.
\]
This proves that the \Holder regularity is preserved for all time.
\end{proof}

\section*{Acknowledgements}
Both authors were partially supported by the NSF of China (Grant No. 11771045, 11871087). 
%
\printbibliography
\vspace{1em}

\quote{\footnotesize C. Khor\\ %
  \textsc{Laboratory of Mathematics and Complex Systems (Ministry of Education), School of Mathematical Sciences, Beijing Normal University, Beijing 100875, People's Republic of China. }  \texttt{C.Khor@bnu.edu.cn}\par  
 \par
  \addvspace{\medskipamount}
    X. Xu\\
   \textsc{Laboratory of Mathematics and Complex Systems (Ministry of Education), School of Mathematical Sciences, Beijing Normal University, Beijing 100875, People's Republic of China. } \texttt{xjxu@bnu.edu.cn} 
} 
\end{document}